\documentclass{amsart}
\usepackage{graphicx,amsmath,amsthm,verbatim,tikz,enumitem,hyperref, amssymb}
\hypersetup{colorlinks,linkcolor={red},citecolor={olive},urlcolor={red}}

\usepackage{mathtools}
\usepackage{xcolor}
\mathtoolsset{showonlyrefs}

\newcommand{\Z}{\mathbb{Z}}

\newcommand{\E}{\mathbf{E}}
\DeclareMathOperator{\var}{var}
\renewcommand{\P}{\mathbf{P}}
\newcommand{\f}{\frac}
\newcommand{\p}{\mu}
\newcommand{\lrl}{\longleftrightarrow}

\newcommand{\ra}{\rightarrow}
\newcommand{\la}{\leftarrow}

\renewcommand{\phi}{\varphi}

\newcommand{\cev}[1]{\reflectbox{\ensuremath{\vec{\reflectbox{\ensuremath{#1}}}}}}
\renewcommand{\b}{\bullet}
\renewcommand{\L}{\cev{\b}}
\newcommand{\R}{\vec{\b}}
\newcommand{\B}{\dot{\b}}

\usepackage{theoremref}

\newtheorem{theorem}{Theorem}
\newtheorem{lemma}[theorem]{Lemma}
\newtheorem{proposition}[theorem]{Proposition}
\newtheorem{corollary}[theorem]{Corollary}

\newtheorem{fact}[theorem]{Fact}

\newtheorem{step}{Step}
\theoremstyle{definition}

\title{Non-universality in clustered ballistic annihilation}
\author[Junge]{Matthew Junge} \email{matthew.junge@baruch.cuny.edu}
\author[Ortiz San Miguel]{Arturo Ortiz San Miguel}
\author[Reeves]{Lily Reeves}
\author[Rivera S\'anchez]{Cynthia Rivera S\'anchez}

\thanks{Junge, Ortiz San Miguel, and Rivera S\'anchez were partially supported by NSF grant DMS-1855516. Part of this research was completed during the 2022 Baruch College Discrete Math REU partially supported by NSF grant DMS-2051026. We are grateful to L.P.\ Arguin for helpful comments}

\begin{document}

\maketitle

\begin{abstract}
In ballistic annihilation, infinitely many particles with randomly assigned velocities move across the real line and mutually annihilate upon contact. We introduce a variant with superimposed clusters of multiple stationary particles. Our main finding is that the critical initial cluster density to ensure species survival depends on both the mean and variance of the cluster size. Our result contrasts with recent ballistic annihilation universality findings with respect to particle spacings. A corollary of our theorem resolves an open question for coalescing ballistic annihilation.
\end{abstract}

\section{Introduction}

\emph{Ballistic annihilation} (BA) is a stochastic spatial system in which particles are placed throughout the real line 
with independent and identically distributed spacings
and proceed to move at independently sampled velocities. Collisions result in mutual annihilation. Interest in annihilating dynamics with ballistic particle trajectories arose as an extremal case of diffusion-limited annihilating systems being studied by physicists and mathematicians in the late 20th century \cite{toussaint1983particle, BL3, bramson1991asymptotic}. 
%A precursor to BA were \emph{diffusion-limited annihilating systems} (DLAS) in which particles follow random walk or Brownian trajectories \cite{bramson1991asymptotic}. 
%Interest arose in BA, since ballistic motion appeared to markedly change the global behavior.

Two-velocity BA, with velocities sampled from $\pm 1$, was introduced by physicists Elskens and Frisch \cite{elskens1985annihilation}. Although species survival regimes are simple to infer, the global behavior of ``flocks" of like-particles is subtle and was only recently worked out in full detail \cite{belitsky1995ballistic, kovchegov2020dynamical}. 
%On the other extreme,  BA with velocities sampled from a continuous distribution. There are some conjectures based on mean-field heuristics and simulations \cite{ben1993decay, trizac2002kinetics, krapivsky2001ballistic}, and a result describing the law for the number of surviving particles in finite systems for the related \emph{bullet process} \cite{bullets2}. Otherwise, not much is known.  
Systems that include a third velocity have proven to be a mathematically rich and challenging step up in complexity. %from two-velocity BA.

The three-velocity setting was introduced by Sheu et.\  al in \cite{sheu1991coagulation}, but with modified collision rules to make it resemble the two velocity setting. Subsequently, Droz et.\ al in \cite{droz95} analyzed the symmetric three-velocity setting with velocities sampled from $-1,0,1$. Velocity 0 particles, which we will refer to as \emph{blockades}, occur with probability $p$. Velocity $+1$ and $-1$ particles, which we will call \emph{right} and \emph{left arrows}, respectively, each occur with probability $(1-p)/2$. 

Let $\theta(p)$ be the probability that a given blockade is never annihilated. By ergodicity, the limiting proportion of surviving blockades converges to $p \theta(p)$. 
So,
\begin{align}
    p_c = \inf \{p \colon \theta(p) > 0\} \label{eq:pc}
\end{align} 
represents the critical initial blockade density for species survival.

%Note that symmetry and ergodicity ensure that all arrows are eventually annihilated almost surely.
%Unlike the two velocity setting, t

Unlike the two-velocity setting, three-velocity BA has multiple collision types: arrow--blockade and arrow--arrow. The rates at which these occur are not obvious, and thus neither is the value of $p_c$. Another challenge is that BA exhibits long-range dependence. This makes it difficult to extrapolate from finite systems and to account for multiple velocities. BA is also sensitive to perturbation. 
%changing the velocity of a single particle velocity can have a cascading effect on collision pairings. 
Changing an arrow to a blockade may increase the lifespans of other arrows. Thus, it is not obvious how to rigorously confirm the intuition that $\theta$ and related quantities are monotone in $p$. This is problematic. For example, it is not a priori obvious that the definition of $p_c$ at \eqref{eq:pc} is equal to $\sup \{ p \colon \theta(p) = 0\}$.

Droz et.\ al \cite{droz95} and, later in more detail, Krapivsky et.\ al \cite{ krapivsky95} worked out the phase-behavior of three-velocity BA and concluded that $p_c=1/4$. However, the derivations were not completely rigorous. Despite some progress towards upper bounds on $p_c$ \cite{bullets, ST, burdinski}, showing that $p_c=1/4$ remained an open problem. Even proving the much weaker statement that $p_c>0$ was a problem widely advertised by Sidoravicius in the mid 2010s.
A breakthrough from Haslegrave, Sidoravicius, and Tournier introduced an exactly solvable approach that proved that $p_c=1/4$ \cite{HST}. In the same work, the authors also worked out finer details such as tail survival probabilities and the ``skyline" of collision types. 

Many of the findings in \cite{HST} are \emph{universal} in the sense that the results hold for any law of particle spacings. For example, $p_c=1/4$ so long as triple collisions almost surely do not occur \cite{burdinski, HST}. Additional universality properties with respect to particle spacings were observed in the  followup work by Haslegrave and Tournier \cite{HT}. Broutin and Marckert discovered that a closely related bullet process with finitely many particles has a universal law governing the number of surviving particles that does not depend on  velocity or spacing laws \cite{bullets2}.

A canonical form of universality is invariance with respect to the average particle density. %The ballistic annihilation literature has exclusively focused on configurations with isolated, single particles. 
It is physically and mathematically natural to allow for clusters of superimposed particles, as is standard in other diffusion-limited annihilating systems \cite{bramson1991asymptotic}. To test the robustness of BA dynamics to the initial particle density, we introduce a variant of BA with random clusters of multiple blockades. We prove that the analogue of the critical value \eqref{eq:pc} depends on more than simply the average initial density of particles. Thus, three-velocity BA lacks this type of universality. To our knowledge, this is a new discovery that was not previously conjectured.

\subsection{Notation}

We let $(x_n)_{n \in \mathbb Z}$ be an ordered sequence of starting locations for particles. To standardize placements, set $x_0 = 0$ and assume that $x_n - x_{n-1}$ are sampled independently according to a continuous distribution with support contained in $(0,\infty)$. 
Let $X$ be a nonnegative integer-valued random variable with probability distribution $\p = (\mu_k)_{k \geq 0}$, and let $f(t) = \E [t^X]= \sum_{k=0}^\infty \mu_k t^k$ be the probability generating function. In an abuse of notation, we will write $\E[\mu]$ and $\var(\mu)$ for the mean and variance of $X$. Take $(X_n)_{n \in \mathbb Z}$ to be independent  and $\p$-distributed. Each site $x_n$ either independently starts with a \emph{cluster} of $X_n$-blockades with probability $p \in [0,1]$, or otherwise contains a single arrow whose velocity is sampled uniformly from $\pm1$. We will sometimes refer to the starting number of blockades in a cluster as the \emph{size} and write \emph{$k$-cluster} to refer to a cluster of size $k$. Blockades are stationary. Left and right arrows move with velocities $-1$ and $+1$, respectively. 

Define \emph{$\p$-clustered ballistic annihilation} to have the just-described starting configuration at time 0. As time evolves, particles move at their assigned velocities. When two arrows collide, both vanish from the system. When an arrow collides with a cluster containing $k\geq 1$ remaining blockades, the arrow vanishes and one blockade is removed from the cluster (so $k-1$ blockades remain). A more formal construction of BA that easily generalizes to include clusters can be found in \cite{HST}. 

We denote the events that a cluster starts at $x_n$ by $\B_n$, or that a left or right arrow starts at $x_n$ by $\L_n$ and $\R_n$, respectively. When $x_n$ contains a cluster, we denote the starting size with a superscript $\B_n^{X_n}$. We will frequently refer to $\B, \L, \R$ as particles. Accordingly, collision events and visits to a location $ u \in \mathbb R$ are specified by
\begin{align}
    \R_m \lrl \L_n &= \{\R_m \text{ and } \L_n \text{ mutually annihilate}\} \\
    \R_m \lrl \L &= \{\R_m \text{ mutually annihilates with an arrow}\}\\
    %\{\R_m \ra \B_n\} &= \{\R_m \text{ mutually annihilates with a blockade in the cluster } \B_n\} \\
    \B_n \la \L_m &= \{\L_m \text{ mutually annihilates with a blockade at $x_n$} \} \\
    \B_n^k \la \L_m &= \{\L_m \text{ mutually annihilates with a blockade at $x_n$, $X_n=k$} \}\\  
    %& \hspace{7.5 cm} \cap \{ X_n=k \}\\
    \B \la \L_n &= \{\L_n\text{ mutually annihilates with a blockade}\}\\
    u \la \L  &= \{u \text{ is visited by a $\L$}\}\\
    u \overset{\; j} \la \L_m &= \{\L_m \text{ is the $j$th $\L$ to arrive to $u$}\}.
\end{align}
The events $\R_m \ra \B_n$, $\R_m \ra \B$, $\R_m \ra \B^k$, $\R \ra u$, and $\R \overset{j}\ra u$ are defined similarly. We denote complements of collision events with $\not \lrl, \not \la,$ and $\not \ra$.  Note that when an arrow hits a cluster we count that as visiting the site, so $\{\B_k \la \L\}\subseteq {\{x_k \la \L\}}$. 

It is often helpful to restrict to a system which only includes particles started in an interval $I \subseteq \mathbb R$. We notate this restriction by including $I$ as a subscript on the event, for example, $(\B_m \la \L_n)_{[x_m, x_n]}$ is the event that $\b_m$ is a blockade that annihilates with a left arrow started at $x_n$ in the process restricted to only the particles in $[x_m,x_n]$.
Unless indicated otherwise, the default is that events are one-sided i.e., restricted to $(0,\infty)$. So, $\P( 0 \la \L) = \P( (0 \la \L)_{(0,\infty)}).$ 

We now define the generalization of $\theta$ from the previous section for $\p$-clustered BA:  
\begin{align}
\theta =\theta(p, \p) := \P( (\R \not \ra 0)_{(-\infty,0)}  \wedge ( 0 \not \leftarrow \L)_{(0,\infty)}).\label{eq:theta}    
\end{align}
It is convenient to instead work with the one-sided complement
\begin{align}
q=q(p, \p) := \P( 0 \leftarrow \L), \label{eq:q}
\end{align}
 so that $\theta = (1-q)^2.$ Define the critical value
\begin{align}
    p_c =p_c(\p) := \inf \{ p \colon \theta(p,\mu) > 0\} = \inf \{ p \colon q(p,\mu) < 1\}.
\end{align}

\subsection{Results}

Our main result is a simple formula for $p_c$ that depends on both the mean and variance of $\p$. We also provide an implicit formula for $q$.
\begin{theorem} \thlabel{thm:main}
For $\p$-clustered BA it holds that
\begin{align}
p_c &= \f 1 {(\E[\mu] + 1)^2 + \var(\mu) }. \label{eq:pcform}
\end{align}
 Moreover, $q$ is continuous, strictly decreasing on $[p_c,1]$, and solves
\begin{align}
    \frac{(1-q)^2}{\left(1-q^2\right) q^2 f'(q)-2 q f(q)+q^2+1} = p\label{eq:qform}
\end{align}
    with $f(q) = \sum_{k=0}^\infty \mu_k q^k$ the probability generating function of $\mu$.
\end{theorem}

From this we obtain four immediate corollaries. The first is that the value $p_c=1/4$ in BA from \cite{HST} is maximal. %among all $\mu$-clustered BA with $\E[\mu]=1$. 

\begin{corollary}
%$p_c$ at \eqref{eq:pcform} is maximized if $\p$ consists of a single atom. For example, 
%$p_c =1/4$ from \cite{HST} with 
The distribution with $\mu_1 =1$ has $p_c=1/4$. For all other $\mu$ with $\mu_1 <1$ and $\E[\mu]=1$, it holds that $p_c < 1/4$.
\end{corollary}

A surprising consequence of  \thref{thm:main} is that there is no phase transition whenever $\mu$ has infinite variance. 
\begin{corollary}
    If $\var(\mu)=\infty$, then $p_c=0$. 
\end{corollary}

The third corollary concerns the rate that $p_c$ goes to 0 as the mean or variance of $\mu$ are augmented. 

\begin{corollary}
Let $\mu^{(n)}$ be a sequence of probability distributions with $\E[\mu^{(n)}] = m_n$ and $\var(\mu^{(n)}) = \sigma^2_n$. If $ \max(m_n , \sigma^2_n) \to \infty$, then $p_c(\mu^{(n)})\sim (m_n^2 + \sigma_n^2)^{-1}.$
\end{corollary}

Lastly, Benitez, Junge, Lyu, Redman, and Reeves studied a coalescing version of ballistic annihilation in which particles sometimes survive collisions \cite{benitez2020three}. The primary interest was determining the analogue of $p_c$ for these systems. However, they were unable to analyze the case in which blockades survive each collision with some fixed probability (see \cite[Remark 5]{benitez2020three}). This is equivalent to $\p$-clustered ballistic annihilation with $\p$ a geometric distribution. Thus, \thref{thm:main} gives the value of $p_c$ in this unsolved case.

\begin{corollary}
Let $\beta \in (0,1)$ and $\mu$ be a geometric distribution with parameter $\beta$, i.e., $\mu_k = (1-\beta)^{k-1}\beta$ for $k \geq 1$. For $\p$-clustered ballistic annihilation it holds that $$p_c = \f 1 { (\beta^{-1} + 1)^2 + (1-\beta) / \beta^2}$$
and, by solving \eqref{eq:qform} for $q$, 
\begin{align}
q(p) &= \frac{\sqrt{p^2 \beta-p^2-p \beta+2 p}-p \beta+\beta-1}{p \beta^2-p \beta+p-\beta^2+2 \beta-1}.\label{eq:geoq}
\end{align}
See Figure~\ref{fig:qgeo} for a plot when $\beta=1/2$.
\end{corollary}

\begin{figure}
    \centering
    \includegraphics[width = 8 cm]{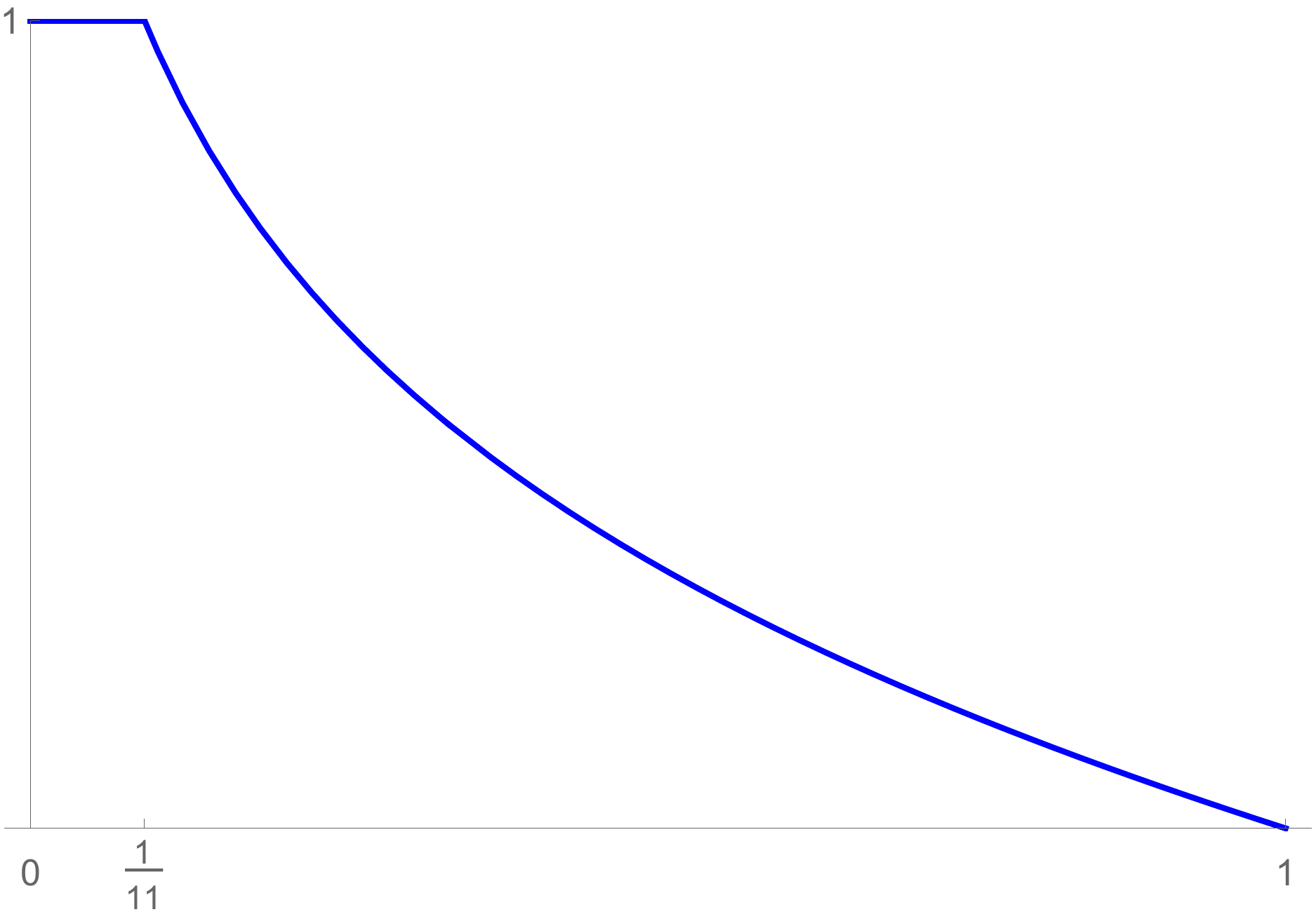}
    \caption{The plot of $q(p)$ from \eqref{eq:geoq} for $\p$-clustered BA with $\mu$ a (mean 2) geometric distribution with parameter $\beta=1/2$. Finding $q$ and $p_c$ resolves an open problem from \cite{benitez2020three}}
    \label{fig:qgeo}
\end{figure}

\subsection{Discussion}
There is no robust general theory that tells us whether or not a given interacting particle system will have a universal phase transition. Rather, universality is typically established on a case-by-case basis using disparate methods. On $\mathbb Z^d$, branching processes, diffusion-limited-annihilating systems, and activated random walk are processes known to have phase transitions that do not depend on the initial particle density  \cite{athreya2004branching, cabezas2018recurrence, rolla2019universality}. The frog model and directed parking processes on $d$-ary trees have phase transitions that depend on more than the average density \cite{aldous2022parking, curien2019phase, bahl2019parking, johnson2019sensitivity}. And, the number of visits to a distinguished site varies monotonically with the concentration of the initial particle placements for these processes on general families of graphs \cite{johnson2018stochastic, bahl2021diffusion}. Additionally, the limiting shape in first passage percolation is also known to be non-universal \cite{BK, Marchand, gorski2022strict}, as are certain mappings associated to spin glass models \cite{collet1983study}.

Our result has a  similar character as the main result of Curien and H\'enard from \cite{curien2019phase}. They studied the parking process on critical Galton-Watson trees whose offspring distribution has variance $\Sigma^2$. A parking spot is placed at each vertex of such a tree and a random number of cars with mean $m$ and variance $s^2$
arrive independently to the vertices. Cars drive towards the root and park in the first spot they encounter. Only one car is allowed per spot, so this reaction can be viewed as ballistic particles (cars) mutually annihilating with stationary blockades (spots). \cite[Theorem 1]{curien2019phase} established that the expected number of cars to reach the root is infinite if and only if $(1-m)^2 - \Sigma^2 (\sigma^2 +m^2 -1) <0$. Like our result, the phase behavior depends (in a surprisingly simple way) on both the mean and variance of the initial configuration. Note that a more complicated formula that depends on the entire distribution of $X$ was observed for parking on binary trees in \cite{aldous2022parking}. 

%A key difference between parking and BA is that when two cars meet in a parking process they continue driving. In BA, collisions between mobile particles results in mutual annihilation. Thus, the reason for the phase behavior in 

It is a priori unclear whether or not $p_c$ should only depend on $\E[\mu]$. On one hand, BA has dynamics similar to the systems considered in \cite{bahl2021diffusion}. So, it is reasonable to expect some sensitivity to the initial density of particles. On the other hand, the mean-field heuristic presented in \cite{droz95} and further clarified in \cite[Section (b)]{krapivsky95} suggests that $p_c$ might be universal. 
%Since arrows close the distance between one another at rate 2, whereas arrows close the distance between blockades at rate 1. 
The explanation in \cite{krapivsky95} assumes that arrow--arrow collisions, on average, occur at twice the rate of blockade--left arrow collisions.  This is ``based on the expectation that
the relative number of annihilation events is proportional to the relative velocities of the collision partners.'' If this ``expectation", which seems to only depend on the relative velocities of particle types, still holds in $\mu$-clustered BA, then the same heursistic would predict universality.
%$P_{+-}$ to  blockade--left arrow collisions occur at twice the rate of b $P_{0-}$ satisfies $$P_{+-}/P_{0-} =2.$$ 

%However, the lack of monotonicity makes proving such a dependence difficult. For example, both \cite{johnson2018stochastic, bahl2021diffusion} depend on a coupling that gives monotonicity for visits to a site after adding a particle. BA lacks this property. 

\thref{thm:main} settles the question. Put concisely, the more volatile $\mu$ becomes, the more space for arrow-arrow collisions, which enhances blockade survival. A more detailed heuristic for why the variance plays a role in the formula for $p_c$ in \thref{thm:main} comes from considering the extreme case in which $\mu_0 = (k-1)/k$ and $\mu_k = 1/k$ with $k$ a large integer. As $\var(\mu) = k-1$ and $\E[\mu] = 1$, \thref{thm:main} implies that $p_c = 1/ (k + 3)$. To see intuitively why this is the correct order, suppose that $0$ contains a $k$-cluster. Let $x_N$ be the next site to the right of $0$ that contains a $k$-cluster. We have $N$ is a geometric random variable with parameter $p/k$. Thus, we expect on the order of $(1-p)k/p$ arrows in $(0,x_N)$ along with some $0$-clusters. The amount of arrows that reach the boundary of $(0,x_N)$ should be comparable to the magnitude of the discrepancy between left and right arrows started in $(0,x_N)$. By the central limit theorem, the discrepancy is on the order of $\sqrt{k/p}$, and so this order of left arrows from $(0,x_N)$ will reach $0$ \cite{elskens1985annihilation}. For these arrows to eliminate a significant portion of the $k$-cluster at $0$, we would need $k \approx \sqrt{k/p}$, equivalently, $p \approx 1/k = 1 / \var(\mu)$. The true evolution is more complicated, but this suggests that $\var(\mu)$ matters and matches the order in our formula from \thref{thm:main}.

\subsection{Proof Overview}
Our proof has three main parts.
Section~\ref{sec:rec} is devoted to proving the recursive equation for $q$ in \thref{prop:qrec}. This is inspired by what was done in \cite{HST}, but instead uses a version of the mass transport principle first observed in \cite{JL} and refined in \cite{benitez2020three}. The basic idea is to partition the event associated to $q$ based on the velocity of $\b_1$.

 An important probability for deriving this recursion is
 $s_k = \P( (0 \la \L) \wedge (\R_1 \ra \B^k))$.  %\\
 %r &= \P( (0 \not \la \L) \wedge (\R_1 \ra \B)). \label{eq:r}
 In \cite{HST}, it was observed that $s_1= (1/2) p q^2$. This comes from the insightful observation that the event measured by $s_1$ has the same probability as the event in which two left particles arrive to $0$ with the time of the first arrival strictly smaller then the time between the first and second arrivals. Two arrivals occur with probability $q^2$ and, even with no explicit knowledge of the arrival time distribution, symmetry and independence give that the first arrival takes less time with probability $1/2$. 
%  In our setting we must compute for $k \geq 1$
%  \begin{align}
%  s_k &= \P( (0 \la \L) \wedge (\R_1 \ra \B^k)) \label{eq:s}. %\\
%  %r &= \P( (0 \not \la \L) \wedge (\R_1 \ra \B)). \label{eq:r}
%  \end{align}

 Computing $s_k$ for $k >1$ in the proof of \eqref{eq:s} is more involved. After applying the mass transport principle, this event partitions into various events in which $k+1$ left arrows arrive to 0 while satisfying non-symmetric spacing requirements. Remarkably, a broader symmetry (see \eqref{eq:syme} makes this case tractable and yields the simple formula $s_k = (1/2) p q^{k+1}$. 
    In the proof of \eqref{eq:r}, we use similar methods to give a relatively simple formula for the companion probability $r_k = \P( (0 \not \la \L) \wedge (\R_1 \ra \B^k))$.
  With these quantities in hand, it is straightforward to obtain \eqref{eq:qrec1}. 
 
  The second part is proving that $q$ is continuous in $p$. The proof closely follows the argument that $\theta$ is continuous in asymmetric three-velocity ballistic annihilation from \cite{JL}. The idea is to prove that $q$ is lower- and upper-semicontinuous (LSC and USC, respectively), and thus continuous. That $q$ is LSC follows from a simple approximation of $q$ from below using finite intervals. See \thref{prop:LSC}. Proving $q$ is USC is equivalent to proving that $\theta$ is LSC. This follows from a characterization of $\theta$ as a supremum of the expected number of surviving blockades minus surviving arrows in finite intervals in \thref{lem:theta}. If this value is positive, then a superadditivity property (see \thref{step:super}) allows for a comparison to a random walk with drift. In Section~\ref{sec:cont} we give the proof and refer to \cite{JL} for some of the more technical details. 
 
 The last step, in Section~\ref{sec:proof}, involves analyzing the recursion from \thref{prop:qrec}. The recursion implies that $0 = (1-q) h(p,q)$
 for an explicit function $h$.
  This tells us that either $q = 1$ or solves $h(p,q) = 0$. We prove that $h(u,1)$ has unique solution $u = p_c$ from \thref{thm:main}. The goal is then to show that $q$, for any $\mu$-clustered BA, resembles the function plotted in Figure~\ref{fig:qgeo}. A priori, it is not obvious how to prove that the roots of $h$ are well-behaved and that $q$ faithfully follows them. 
  %We first deduce that there is only one candidate level curve of $h=0$ in $[0,1]^2$. At this point, t
  The continuity of $q$ proven in Section~\ref{sec:cont} is crucial for ruling out the pathology that $q$ jumps between 1 and solutions to $h=0$. This is a novel and seemingly powerful approach. Both \cite{HST} and \cite{benitez2020three} relied on an explicit characterization of the solutions to $h(p,q)=0$, which came from the quadratic formula, to infer $p_c$. The increased generality of our method seems promising for studying more general BA processes.

\subsection{Further Questions}
It would be interesting to generalize to the setting with $\nu$-distributed clusters of arrows. There should be no problem extending the argument from \thref{thm:cont} that $q$ is continuous. However, computing the analogues of $s_k$ and $r_k$ appears challenging. In our setting, we were able to control the various ways in which an arrow collides with a blockade cluster. Since arrow clusters move, it is more difficult to obtain formulas for arrow--arrow collision events.
%For example, if we allow for arrow clusters of size 1 and 2, then $\P(\R_1^1 \lrl \L^2)$ does not have an obvious relationship with $\P(\R_1^1 \lrl \L^1)$ and similarly for subevents of $\P(\R_1^2 \lrlr \L)$.
In line with monotonicity results from \cite{bahl2021diffusion}, we conjecture that $p_c$ increases with the mean and variance of $\nu$. 

Another followup question is determining if the behavior of $\mu$-clustered BA at $p=p_c$ is universal. \cite[Theorem 3]{HST} proved that $\P_{p_c}(\B_1 \text{ survives beyond time $t$})$ and $\P_{p_c}(\R_1 \text{ survives beyond time $t$})$ are on the order of $t^{-2/3}.$ This is proved by deriving a recursive relationship for the generating function of the index of the first particle to arrive to $0$, and then conducting singularity analysis. We did not pursue this here, but it is a possibly tractable next step in exploring the universality of BA.

Studying the phase transition for diffusion-limited annihilating systems in which all collisions result in mutual annihilation is an interesting direction. For example, consider the variant of BA in which, rather than following ballistic trajectories, left and right arrows performed random walk. The literature \cite{BL3, cabezas2018recurrence, johnson2020particle} has focused on the setting in which mobile particles do not interact. To the best of our knowledge, proving that there is a phase transition in systems with diffusive, mutually annihilating mobile particles is open. 

\section{Recursion} \label{sec:rec}
The goal of this section is to prove the following recursive formula.

\begin{proposition} \thlabel{prop:qrec}
%Let $q$ be as defined at \eqref{eq:q} and $f(t) = \sum_{k\geq 0} \mu_k t^k$ be the probability generating function of $\mu$. 
%It holds that 
\begin{align}
q= \frac{1-p}{2} + pqf(q) + s + q\left(\frac{1-p}{2} - s - r \right) \label{eq:qrec1}
\end{align}
with
\begin{align}
s &:= \P((0 \la \L) \wedge (\R_1 \ra \B))= \frac{pq^2}{2} f'(q)  \label{eq:s}   \\
 r &:= \P((0 \la \L) \wedge (\R_1 \ra \B))= \frac{pq{\left(q^{2}f'\left(q\right) - q f'\left(q\right) - f\left(q\right) + 1\right)}}{1- q}.  \label{eq:r}
\end{align}
% Formulas for $s$ and $r$ in terms of $q$ are  at \eqref{eq:sform} and \eqref{eq:rform}, respectively.
\end{proposition}
\begin{proof}[Proof of \eqref{eq:qrec1}]
%We will prove \eqref{eq:s} and \eqref{eq:r} momentarily. For now, we focus on establishing \eqref{eq:qrec1}. 
We partition $q$ in terms of the velocity of the first particle
\begin{align}
q = \P((0\la \L) \wedge \L_1) + \P((0 \la \L) \wedge \B_1) + \P((0 \la \L) \wedge \R_1) 
\label{eq:q1}
\end{align}
and will provide a formula for each summand.
% We calculate each of these parts separately. If $\L_1$, then the $\P(0 \la \L) = 1$ since the first particle cannot collide. Thus,
It is immediate that
\begin{align}
\P((0\la \L) \wedge \L_1) = \frac{1-p}{2}. \label{eq:t1}
\end{align}
For the second summand, we further partition on the size of $\B_1$ to write
\[
\P((0 \la \L) \wedge \B) = \sum_{k=0}^\infty \P( (0 \la \L) \wedge \B_1 \wedge (X_1=k) ).
\]
If $X_1 = k$, then $k+1$ left arrows must arrive at $x_1$ in order for $0$ to be visited. This happens if and only if the $j$th left arrow to arrive reaches the starting location of the $(j-1)$th left arrow to arrive for $j=2,\hdots,k+1$. By similar reasoning as \cite[Lemma 7]{HST}, each of these arrivals is conditionally independent and has probability $q$. Thus, for $k \geq 0$ we have
$$\P( (0 \la \L) \wedge \B_1 \wedge (X_1=k )) = p \cdot p_k q^{k+1}.$$
Summing over $k$ gives 
\begin{align}
 \P((0 \la \L) \wedge \B_1)= pqf(q). \label{eq:t2}
\end{align}
%Then, using the fact that $x_{n-1} \la \L_n$ and $x_n \la \L_{n+1}$ implies $x_{n-1} \la \L_{n+1}$, we have that
% \[
% \P((0 \la \L) \wedge \B^k_1) = p_kp\prod_{i=1}^{k+1} \P(x_{i-1} \la \L)_{[x_{i-1}, \infty)} = p_kpq^{k+1}
% \]
% Thus,

%Now, if $\B_1$, all particles that collide at $\B_1$ must be $\L$. Thus, when $X=k$, for $(0 \la \L)$ to occur, k+1 left arrows must arrive at the site of $\B_1^k$. Thus,

%Lastly, the case with $\R_1$ is more involved. It is hard to find $\P(\R_1 \lrl \L)$. Thus, we take advantage of the fact that $\R_1$ must annihilate and write
A similar argument as \cite[Lemma 3.3]{ST} implies that all arrows are eventually annihilated. Since $\P(\R_1 \ra \B)= s+r$, we may write
\begin{align}
\P(\R_1) = \frac{1-p}{2} &= \P(\R_1 \ra \B) + \P(\R_1 \lrl \L)\\
&= s+r + \P(\R_1 \lrl \L). \label{eq:R1}
\end{align}
For $0$ to be visited on the event $\{\R_1\}$, the particle $\R_1$ must first be annihilated. We partition on the collision type: 
\begin{align}
\P((0\la \L) \wedge \R_1) &= \P((0 \la \L) \wedge (\R_1 \ra \B)) +  \P((0 \la \L) \wedge (\R_1 \lrl \L)) \\
&=s + q \P(\R_1 \lrl \L) \label{eq:lrl}\\
&= s + q\left( \frac{1-p}{2} - s - r \right) \label{eq:R1'}.
\end{align}
The equality at \eqref{eq:lrl} follows from the definition of $s$ and the fact that
$\P(0 \la \L \mid \R_1 \lrl \L)= q$.
This fact follows from the observation that conditional on $(\R_1 \lrl \L_j)$ for some $j>1$, $(0 \la \L)$ occurs if and only if $(x_j \la \L)_{(x_j, \infty)}$, which has probability $q$. The move to \eqref{eq:R1'} then uses \eqref{eq:R1}.
%\[= q(\P(\R_1 \lrl \L)) = s + q\left( \frac{1-p}{2} - s - r \right)
%\]
%... + \P((0 \la \L) \wedge (\R_1 \lrl \L)) = ... + q(\P(\R_1 \lrl \L))
Combining \eqref{eq:t1}, \eqref{eq:t2}, and \eqref{eq:R1'} in \eqref{eq:q1} gives \eqref{eq:qrec1}.
% ,
% $$q= \frac{1-p}{2} + pqf(q) + s + q\left(\frac{1-p}{2} - s - r \right)$$
\end{proof}

Next, we will prove the formulas for $s$ and $r$ at \eqref{eq:s} and \eqref{eq:r}, respectively. These require the use of a Mass Transport Principle based on translation invariance.
%, which shows that the probability of an event occurring on the interval $[x_a,x_b]$ is equal to the probability of the same event occurring on the $[x_a + \ell, x_b + \ell]$ for any $\ell \in \mathbb{R}$. 
\begin{proposition}[Mass Transport Principle] \thlabel{prop:mtp}
    Define a non-negative random variable $Z(m,n)$ for integers $m,n \in$ Z such that its distribution is diagonally invariant under translation, i.e., for any integer $\ell$, $Z(m+\ell,n+\ell)$ has the same distribution as $Z(m,n)$. Then for each $m \in \Z:$
    \begin{align}
        \E \sum\limits_{n \in Z} Z(m,n) = \E \sum\limits_{n \in Z} Z(n,m).
    \end{align}
\end{proposition}
\begin{proof}
Fubini's theorem and translation invariance give
\begin{align}
    \E \sum\limits_{n \in Z} Z(m,n) &=  \sum\limits_{n \in Z} \E [Z(m,n)] \\
    &=\sum\limits_{n \in Z} \E [Z(2m-n,m)]
    = \sum\limits_{n \in Z} \E [Z(n,m)] = \E \sum\limits_{n \in Z} Z(n,m).
\end{align}
\end{proof}

% \begin{lemma} \thlabel{lem:s} 
% $$s := \P((0 \la \L) \wedge (\R_1 \ra \B)) = \f{pq^2}{2} f'(q).$$
% \end{lemma}
\begin{proof}[Proof of \eqref{eq:s}]
Let $s_k = \P(( 0 \la \L) \wedge (\R_1 \ra \B^k))$ so that $s = \sum_{k=0}^\infty s_k$. We will use the mass transport principle to relate the event associated to $s_k$ to one that involves $k+1$ arrows arriving to the site containing a $k$-cluster. To this end, define 
%The argument depends on considering how many particles collide with the blockade from the left and right respectively. To prove that there are always enough particles to do so, we use the Mass Transport Principle.
\begin{align}
    Z_k^j(a,b) &= \sum_{c \in \mathbb Z}\Big[ \mathbf{1}\{\B_b^k \wedge (\R_a \xrightarrow[]{j} x_b)_{[x_a, x_b)} \wedge (x_b \xleftarrow[]{k+1-j} \L_c)_{(x_b, x_c]} \\
    & \qquad \qquad\qquad \qquad\qquad \qquad\qquad \qquad \wedge (x_b - x_a < x_c - x_b) \} \Big ]
\end{align}
%
% \[
% Z_k^j(a,b,c) = \mathbf{1}\left\{\B_b^k \wedge (\R_a \xrightarrow[]{j} x_b)_{[x_a, x_b)} \wedge (x_b \xleftarrow[]{k+1-j} \L_c)_{(x_b, x_c]} \wedge (x_b - x_a < x_c - x_b) \right\}
% \] 
for $a,b,j, k \in \mathbb Z$. 

Observe that 
\begin{align}
    s_k^j := \P((\R_1 \overset{\; j} \ra \B^k) \wedge (0 \la \L)) =\E\sum_{b\in \mathbb{Z}} Z_k^j(1,b) .
\end{align}
Define $\vec D_j$ to be the starting distance from $x_1$ of the $j$th particle to arrive to $x_1$ in the process restricted to particles in $(-\infty, x_1)$. We set $\vec D_j = \infty$ whenever fewer than $j$ particles ever visit $x_1$. Define $\cev D_j$ similarly, but on $(x_1,\infty)$. 
By \thref{prop:mtp} and independence, $s_k^j$ is equal to
\begin{align}
 \E \textstyle{\sum_{a \in \mathbb{Z}}} Z_k^j(a,1) &= \P(\B_1^k) \P((\R \overset{j} \ra x_1)_{(-\infty, x_1)} ) \\
 & \qquad\qquad\times \P((x_1 \xleftarrow[]{k+1-j} \L)_{(x_1,\infty)}) \P(\vec D_j < \cev D_{k+1-j}) \\
 &=p\cdot p_kq^jq^{k+1-j}\P(\vec D_j < \cev D_{k+1-j}). \label{eq:sjk}
\end{align}
Since $s_k = \sum_{j=1}^k s_k^j$, \eqref{eq:sjk} gives
$$s_k = p \cdot p_k q^{k+1}  \sum_{j=1}^k  \P(\vec D_j < \cev D_{k+1-j} ).$$

If $k$ is even, then grouping summands gives
\begin{align}
\sum_{j=1}^k  \P(\vec D_j < \cev D_{k+1-j} ) &=  \sum_{j=1}^{k/2} \left[ \P(\vec D_j < \cev D_{k+1-j} ) +  \P(\vec D_{k+1-j} < \cev D_{j} ) \right] \\
&= \f k2 \label{eq:syme}.
\end{align}
We have $\P(\vec D_j < \cev D_{k+1-j} )+\P(\vec D_{k+1-j} < \cev D_{j} )=1$, because $\vec D_m$ and $\cev D_m$ are continuous, independent, and identically distributed random variables.
Using similar reasoning, if $k=2m+1$ is odd, then we can write $\sum_{j=1}^k  \P(\vec D_j < \cev D_{k+1-j} )$ as
\begin{align}
  \sum_{j=1}^{m} \left[ \P(\vec D_j < \cev D_{k+1-j} ) +  \P(\vec D_{k+1-j} < \cev D_{j} ) \right] + \P(\vec D_{m +1} < \cev D_{m + 1}),
\end{align}
  which equals $m + (1/2) = k/2$.
%
% By \thref{prop:mtp}
% \[
% \E\sum_{b\in \mathbb{Z}} Z_k^j(1,b) = \E \sum_{a \in \mathbb{Z}} Z_k^j(a,1) = p_mq^jq^{m+1-j}\P(\vec D_j < \cev D_{m+1-j}),
% \]
% where $s_k^j = \P((\R_1 \overset{\; j} \ra \B^k) \wedge (0 \la \L))$.
% Now, we note that there are exactly $k-1$ disjoint ways that $((0 \la \L) \wedge (\R_1 \ra \B^k))$ can occur. These are when $a \in [1,k]$ particles collide from the left and $b =k-a$ from the right and a left moving particle, $\L_z$ arrives at 0. This can only happen when $d_a = \textrm{dist}(\R_1, \B^k) < \textrm{dist}(\B^k, \L_z) = d_b$. Thus,
%
% \[
% s_k = \sum_{j=1}^k s_k^j = \sum_{a+b=k-1} \P(d_{a} < d_{b}) \P(\R_1 \ra \B^k) \P(\R \ra \B^k)^{a-1} \P(\B^k \la \L)^b \P(0 \la \L_z).
% \]
% Then,
% \[
% s_k = p_k\sum_{a+b = k-1} \P(d_{a} < d_{b}) pq^{a+b+2}.
% \]
% Now consider the following pair of cases. First, when $a = \alpha$, $b= \beta$ and then when $a = \beta, b = \alpha$. Then, by symmetry, $\P(d_\alpha < d_\beta) = 1 - \P(d_\beta < d_\alpha)$. Thus, if $k$ is even, then
% \[
% \sum_{a+b = k-1} \P(d_{a} < d_{b})q^{k+1} = \frac{k}{2}q^{k+1}.
% \]
% Otherwise, when $k$ is odd, we have that all cases where $a\neq b$ have a unique and distinct case that meets an opposite distance inequality.
% \[
% \sum_{a+b = k-1} \P(d_{a} < d_{b})q^{k+1} = \frac{k-1}{2}q^{k+1} + \P(d_{a=b}< d_{b=a})q^{k+1} = \frac{k}{2}q^{k+1}.
% \]
% Thus, by symmetry, we have, that for all $k$
% \[
% s_k = p_k\sum_{a+b = k-1} \P(d_{a} < d_{b})q^{k+1} = p_k\frac{k}{2}q^{k+1}.
% \]
Hence, $s_k = p \cdot p_k q^{k+1} (k/2)$. Summing gives 
\[
s =  \sum_{k=0}^\infty s_k = \frac{pq^2}{2} \sum_{k=0}p_kkq^{k-1}  = 
\frac{pq^2}{2}f'(q).
\]
\end{proof}

% \begin{lemma} \thlabel{lem:r} 
% $$r := \P((0 \not \la \L) \wedge (\R_1 \ra \B)) = \f{pq{\left(q^{2}f'\left(q\right) - q f'\left(q\right) - f\left(q\right) + 1\right)}}{1-q}.$$
% \end{lemma}

\begin{proof}[Proof of \eqref{eq:r}]
Let $r_k = \P( (0 \not\la \L) \wedge (\R_1 \ra \B^k))$ so that $r = \sum_{k=0}^\infty r_k$.
%
%There are exactly $i$ disjoint cases where this event can happen. These are when $i \in [1,k]$ right arrows hit the same $\B$. Then, each of these cases has $i-j$ sub-cases. Namely, when exactly $j\in [0,k-i]$ particles hit the blockade from the right.
As we did for the proof of \eqref{eq:s}, we apply the Mass Transport Principle with new indicators
\begin{align}
W_k^{i,j}(a,b) &= 
\textstyle \sum_{c \in \mathbb Z} \mathbf{1}\{\B_b^k \wedge (\R_a \xrightarrow[]{i} x_b)_{[x_a, x_b)} \wedge (x_b \xleftarrow[]{j} \L_c)_{(x_b, x_c]} \wedge (x_c \not \la \L)_{(x_c, \infty)}\} 
\end{align}
for $i,j,k,a,b \in \mathbb Z$.
% \ind{ (\R_a \xrightarrow[]{i} \B_b^k)_{[x_a,x_b]} \wedge (D_j < \infty) \wedge (D_{j+1} = \infty)}
% \]
% where $D_\ell := \textrm{dist}(x_b, x_\ell)$ where $\L_\ell$ is the $\ell^{\textrm{th}}$ particle to arrive at $x_b$ from the right and begins at $x_\ell$.
Let $(\R_1 \overset{i}\ra \B^k \overset{\;\;\;\;j^*}\la \L)$ denote the event that $\R_1$ is the $i$th right arrow to annihilate with a $k$-cluster and \emph{exactly} $j$ left arrows visit that same $k$-cluster. Observe that for $i + j \leq k$ with  $i \neq 0$, we have
$$r^{i,j}_k := \P( (0 \not \la \L) \wedge (\R_1 \overset{i} \ra \B^k \overset{\;\;\;\;j^*} \la \L)) = \E \sum_{b \in \Z} W_k^{i,j}(1,b).$$
By \thref{prop:mtp} and independence, 
\begin{align}
r^{i,j}_k=\E \sum_{a \in 
\Z} W_k^{i,j}(a,1) &= \P(\B_1^k) \P( (\R \overset{i} \ra x_1 )_{(-\infty, x_1)} )\P( (x_1 \overset{j} \la \L)_{(x_1,\infty)} ) \\
&\qquad \qquad \qquad \qquad \times\P(\cev D_{j+1} = \infty \mid \cev D_j < \infty)\\
&=p\cdot p_k q^i q^j (1-q).
\end{align}
We then have
$$r_k = \sum_{i=1}^{k}\sum_{j=0}^{k-i} r_k^{i,j}= p\cdot p_k\sum_{i=1}^k \sum_{j=0}^{k-i} q^{i+j}(1-q).$$ 
Applying the formula $\sum_{i=0}^m a^i = ({1- a^{m+1}})/({1- a})$ twice, gives 
%
% \[
% r_k = \sum_{i,j}\sum_{b \in \mathbb{Z}} E(\ind {(\R_1 \xrightarrow[]{i} \B_b) \wedge (D_j < \infty) \wedge (D_{j+1} = \infty) } = \sum_{i,j}(\sum_{b\in \mathbb{Z}} E(Z(1,b))
% \]
% \[
% = \sum_{i,j} \sum_{a\in \mathbb{Z}} E(Z(a,1)) = \sum_{i,j} \sum_{a\in \mathbb{Z}} E(\ind{(\R_a \xrightarrow[]{i} \B_1^m) \wedge (D_j < \infty) \wedge (D_{j+1} = \infty) }).
% \]
% Then, we derive the claimed formula by factoring to get a geometric series and separating the sum to get a second geometric series.
\[
r_k = p\cdot p_k\frac{q \left(k q^{k+1}-k q^k-q^k+1\right)}{1-q}.
\]
Hence,
\[
r = \sum_{k=0}^\infty r_k = \frac{pq{\left(q^{2}f'\left(q\right) - q f'\left(q\right) - f\left(q\right) + 1\right)}}{1-q}.
\]
\end{proof}

\section{Continuity} \label{sec:cont}

The goal of this section is to prove that $q$ is continuous in $p$ by proving that it is both upper and lower semi-continuous. We begin by recalling these definitions and stating a few classical facts. A function $\varphi$ is \emph{upper semi-continuous} (USC) at each $p_0 \in [0,1]$ if and only if ${\limsup _{p\to p_{0}}\varphi(p)\leq \varphi(p_{0})}$. It is \emph{lower semi-continuous} (LSC) at each $p_0 \in [0,1]$ if and only if it holds that ${\liminf_{p\to p_{0}}\varphi(p)\geq \varphi(p_{0})}$.
 Rather than working directly with the definition, we will apply the following properties. See \cite{hobson1927theory} for proofs.

\begin{fact} \thlabel{fact:SC}
The following hold.
%and $y>f(p_0)$ it holds that $\varphi(p) < y$ for all $p$ in some neighborhood of $p_0$. Analogously,  $\varphi$ is \emph{lower semi-continuous} if $\varphi(p) > y$ for all $p$ in some neighborhood $p_0$. We abbreviate these properties as USC and LSC, respectively, and recall some basic facts.
%Consider functions $\varphi, \psi$ and $\varphi_n$ defined on $[0,1]$.
\begin{enumerate}[label = (\alph*)]
    \item $\varphi$ is continuous if and only if $\varphi$ is USC and LSC. \label{fact:cont}
    \item If there exists a sequence of LSC functions $\varphi_n$ with $\varphi_n \uparrow \varphi$, then $\varphi$ is LSC. \label{fact:inc}
    \item If $\varphi(p) = \sup_n (\varphi_n(p))$ with $\varphi_n$ LSC, then $\varphi$ is LSC. \label{fact:sup}
    \item If $\varphi_1$ and $\varphi_2$ are LSC, then $\max(\varphi_1,\varphi_2)$ is LSC. \label{fact:max}
    \item $\varphi$ is LSC if and only if $-\varphi$ is USC. \label{fact:neg}
%    \item If $\varphi$ and $\psi$ are both LSC and nonnegative, then $\varphi \cdot \psi$ is LSC. Similary, if both are USC and nonnegative, then the product is USC. \label{fact:product}
    \item If $\psi$ is continuous and $\varphi$ is LSC, then $\psi\circ \varphi$ is LSC. Similarly, if $\varphi$ is USC, then $\psi \circ \varphi$ is USC. \label{fact:comp}
    \item If $\varphi$ and $\psi$ are both LSC or USC, then so is $\varphi + \psi$. \label{fact:sum}
\end{enumerate}
\end{fact}
That $q$ is LSC follows almost immediately from its definition.
\begin{proposition} \thlabel{prop:LSC}
$q$ is LSC for $p \in [0,1]$.
\end{proposition}

\begin{proof}
The events $Q_n= \{(0 \la \L)_{(0,x_n)}\}$ involve finitely many particles. So, $\P(Q_n)$ are finite degree polynomials in $p$, and thus continuous in $p$. Moreover, $Q_n \subseteq Q_{n+1}$, thus the $\P(Q_n)$ are increasing in $n$. Since $q = \lim_{n\to \infty} \P(Q_n)$, it follows from \thref{fact:SC} \ref{fact:inc} that $q$ is LSC. 
\end{proof}

We next aim to prove that $q$ is USC. This is more difficult and involves an indirect characterization of $\theta = (1-q)^2$ that takes a supremum over functionals of configurations with only finitely many particles. Let $\dot N(j,k)$ be the number of blockades that survive in ballistic annihilation restricted to the particles in $[x_j,x_k]$. Similarly, let $\cev N(j,k)$ and $\vec N(j,k)$ count the number of surviving left and right arrows. Define the random variables that track the difference between the number of surviving blockades and arrows in the process restricted to only the particles in $[x_j, x_k]$:
\begin{align}
    W(j,k) = \dot N(j,k) - \cev N(j,k) - \vec N(j,k) \label{eq:W}.
\end{align}

\begin{lemma} \thlabel{lem:W}
$n^{-1}\E_p[W(1,n)]$ is continuous in $p$ for all $n \geq 1$. 
\end{lemma}

\begin{proof}
The random variables $W(1,n)$ involve only finitely many particles. Thus, $\E_p[W(1,n)]$ is a finite degree polynomial in $p$ and is continuous. 
\end{proof}

\begin{lemma} \thlabel{lem:theta}
$\theta =  \max \left( 0 ,  \sup_{n \geq 1} n^{-1} \E_p [W(1,n)]\right)$ for all $p \in [0,1]$. 
\end{lemma}

\begin{proof}
The proof has four steps. Fortunately, it requires little modification from the blueprint developed in \cite{JL}. We explain the basic idea of each step and refer the reader to the appropriate reference.

\begin{step} \thlabel{step:super}
For all integers $j<k< \ell$ it holds that $W(j,\ell) \geq W(j,k) + W(k+1, \ell)$.
\end{step}

\begin{proof}
 This superadditivity property is proven in \cite[Lemma 15]{benitez2020three} for a more general variant of ballistic annihilation in which particles sometime survive collisions. The basic idea is that surviving arrows from the restrictions to $[x_j,x_k]$ and $[x_{k+1}, x_\ell]$ have a non-decreasing effect on $W(j,\ell)$. Surviving arrows either destroy other surviving arrows, which augments $W(j,\ell)$. Or, surviving arrows destroy blockades, which may cause a chain reaction, but, regardless, the effect is worst-case neutral on $W(j,\ell)$. The argument does not change if multiple blockades are present at a site. 
\end{proof}

\begin{step} \thlabel{step:arrow}
$\lim_{k \to \infty} k^{-1} {\cev N(1,k)}  = 0 = \lim_{k \to \infty} k^{-1} \vec N(1,k)$. 
%and $\lim_{k \to \infty} k^{-1} \dot N(1,k) = \theta$. 
\end{step}

\begin{proof}
%The equation for the limits involving $\cev N$ and $\vec N$ are
This is proven in \cite[Proposition 12]{JL} for asymmetric ballistic annihilation. It is much simpler to deduce for symmetric systems. The strong law of large numbers gives that the limits equal the probability an arrow is never annihilated. \cite[Lemma 3.3]{ST} observes that this quantity must be zero. Otherwise, ergodicity and symmetry  imply the contradiction that there is a simultaneously a positive density of surviving left and right arrows.  The same reasoning applies with the possibility of multiple blockades at a single site. 
\end{proof}

\begin{step} \thlabel{step:theta}
Let $N_{\mathbb R}(1,k)$ denote the number of blockades that survive in $[x_1,x_k]$ in ballistic annihilation with all particles in $\mathbb R$ present. If $\theta >0$, then 
$$\lim_{k \to \infty} k^{-1} \dot N(1,k) = \theta = \lim_{k \to \infty} k^{-1} N_{\mathbb R}(1,k).$$
\end{step}

\begin{proof}
This is proven in \cite[Proposition 12]{JL}. It follows from the definition of $\theta$ and the strong law of large numbers that $\lim_{k \to \infty} k^{-1} N_{\mathbb R}(1,k) = \theta$. So, it suffices to prove that $$\lim_{k \to \infty} k^{-1} [\dot N(1,k) - \dot N_{\mathbb R}(1,k)] = 0.$$
%They require the additional hypothesis that $\theta>0$ to ensure that all arrows are eventually annihilated. As observed in the previous step, the symmetric setting always has this property. 
First, observe that blockade survival is a decreasing event as the interval of restriction is expanded. So, $\dot N(1,k) - \dot N_{\mathbb R}(1,k) \geq 0$. 
From there, the main idea is that at most a geometric random variable with parameter $q$, call it $R_k$, of the surviving blockades in ballistic annihilation restricted to $[x_1, x_k]$ are removed from right arrows entering at $x_1$, and the same for an independent and identically distributed geometric random variable of left arrows entering at $x_k$, call it $L_k$. So, $\dot N(1,k) - N_{\mathbb R}(1,k) \leq L_k + R_k$ Since these random variables have exponential tails, it is easy to infer from the Borel-Cantelli lemma that $\lim_{k \to \infty} k^{-1} [R_k+L_k] = 0$ almost surely. 
\end{proof}

\begin{step} 
Let $\theta_0 := \max\left(0, \sup_{k \geq 1} k^{-1} \E_p[ W(1,k) ] \right)$. It holds that $\theta = \theta_0$. 
\end{step}

\begin{proof}
The proof is similar to \cite[Lemma 10]{JL}. First, we prove that $\theta \leq \theta_0$. Combining \thref{step:arrow}, \thref{step:theta}, and Fatou's lemma gives
$$\theta = \lim_{k \to \infty} k^{-1} W(1,k) = \E_p \left[ \liminf_{k \to \infty} k^{-1} W(1,k)  \right] \leq \liminf_{k \to \infty} k^{-1} \E_p[W(1,k) ] \leq \theta_0.$$

Next, we show that $\theta \geq \theta_0$. This is immediate when $\theta_0=0$, so suppose that $\theta_0 >0$. Then, there is an integer $k$ with $\E_p[W(1,k)]>0$. Letting $K_m = km$ for $m \geq 0$, we see that $S_n := \sum_{m=0}^{n-1} W(K_m+1,K_{m+1})$ is a random walk with positive drift. The law of large numbers gives that $S_n > 0$ for all $n \geq 1$ with positive probability. \thref{step:super} implies that 
\begin{align}
    W(1,K_n) \geq S_n \qquad \forall n \geq 1. \label{eq:Sn}
\end{align}
This is enough to deduce that $0$ is never visited with positive probability, which gives $\theta >0$. See the proof of \cite[Lemma 10]{JL} for more details.

Now that we have $\theta >0$, it suffices to prove that $\theta > \delta$ for arbitrary $\delta \in (0,1)$ with $\delta < \theta_0$. Let $k\geq 1$ be such that $k^{-1} \E_p[W(1,k)] > \delta$.  \thref{step:arrow} and \thref{step:theta} imply that $\theta = \lim_{n \to \infty} \f 1 n W(1,n).$
Multiplying by $n/n$, applying \eqref{eq:Sn} and then the strong law of large numbers gives 
$$\theta = \liminf_{n \to \infty} \f{n}{K_n} \f 1 n W(1,K_n) \geq \liminf_{n \to \infty} \f n {K_n} \f 1 n S_n = k^{-1} \E_p [W(1,k)] > \delta$$
as desired. 
\end{proof}

\end{proof}

\begin{proposition} \thlabel{prop:USC}
$q$ is USC for $p \in [0,1]$.
\end{proposition}

\begin{proof}
It follows that $\theta$ is LSC from Lemmas \ref{lem:W} and \ref{lem:theta} along with \thref{fact:SC} \ref{fact:sup} and \ref{fact:max}. Since $\theta = (1-q)^2$, we have $q = 1- \sqrt{ \theta}.$ \thref{fact:SC} \ref{fact:neg} and \ref{fact:comp} imply that $- \sqrt{\theta}$ is USC. Since $1$ is USC, $q$ can be expressed as the sum of two USC functions and by \thref{fact:SC} \ref{fact:sum} is USC.
\end{proof}

\begin{theorem} \thlabel{thm:cont}
$q$ is continuous for $p \in [0,1]$.
\end{theorem}
\begin{proof}
This follows immediately from Propositions \ref{prop:LSC} and \ref{prop:USC} along with \thref{fact:SC} \ref{fact:cont}.
\end{proof}

\section{Proof of \thref{thm:main}} \label{sec:proof}

\begin{proof}
%In what follows we will sometimes write $q(p)$ to emphasize that $q$ is a function of $p$. 
Subtracting $q$ from both sides of \eqref{eq:qrec1} in \thref{prop:qrec} gives $0 = g(p,q)$ with
$g\colon [0,1]^2 \to \mathbb R$ defined as
\begin{align}
g(u,v):= \frac{u \left(1-v^2\right) v^2 f'(v)+2 u v f(v)-u v^2+u+v^2-2 v+1}{2 (1-v)}\label{eq:g}.
\end{align}
Let $h(u,v) = g(u,v)/(1-v)$ so that \thref{prop:qrec} implies 
\begin{align}
0 = (1-q) h(p,q).\label{eq:h}
\end{align}
The goal is to show that $(p,q)$ solves $1-v = 0$ for $p \leq p_c$ and transitions to solving $h(u,v) = 0$ for $p \geq p_c$ as depicted in Figure~\ref{fig:qgeo}.

Inspecting \eqref{eq:g}, we see that $h(u,v)$ is linear in $u$. Solving $h(u,v)=0$ yields
\begin{align}
u=\frac{(1-v)^2}{\left(1-v^2\right) v^2 f'(v)-2 v f(v)+v^2+1} =: F(v).\label{eq:F}
\end{align}
Thus, 
\begin{fact} \thlabel{fact:!}
If $h(u,v)=0$, then $u = F(v)$.
% For each $v \in [0,1]$ there is at most one $u$, in particular $u=F(v)$ when it exists, solving $h(u,v) =0$. 
\end{fact}

Using L'Hospital's rule twice and basic generating function properties %---$f(1) =1$, $f'(1) = \E[X]$, and $f''(1) = \E[X(X-1)]$---
$$\lim_{v \to 1} F(v)= \frac{1}{f(1) + 3f'(1) + f''(1)} =\frac{1}{(1+\E[X])^2+\var(X)}=:p_*.$$
By \thref{fact:!}, 
\begin{fact}\thlabel{fact:!!}
$(u,v)=(p_*,1)$ is the unique solution to $1-v=0=h(u,v)$.
\end{fact}

Since $q$ is continuous (\thref{thm:cont}) with $q(1) = 0$, it follows from \eqref{eq:h} and \thref{fact:!!} that $(p_*,1)$ is the only point at which $(p,q)$ can continuously transition from solving $1-v =0$ to solving $h(u,v) = 0$. So,

\begin{fact} \thlabel{fact:h}
If $p \geq p_*$, then $h(p,q(p))=0$.
\end{fact}

Combining \thref{fact:!} and \thref{fact:h} gives
\begin{fact} \thlabel{fact:F}
$p=F(q(p))$ for $p \geq p_*$.
\end{fact}

\thref{fact:F} says that $F$ is a left inverse of $q$. It is an elementary exercise in analysis that this and continuity of $q$ imply that

\begin{fact} \thlabel{fact:<}
 $q$ is continuous and strictly decreasing for $p \geq p_*$.
\end{fact}

\noindent \thref{fact:F} and \thref{fact:<} (along with \thref{thm:cont}) imply \eqref{eq:qform} in \thref{thm:main}. 

It remains to prove that $p_c = p_*$ as claimed at \eqref{eq:pcform}. 
Suppose that $p> p_*$.  \thref{fact:h} implies that $h(p,q(p))=0$. \thref{fact:!!} ensures that $q(p_*)=1$. \thref{fact:!} requires that $q(p) \neq 1$. Since $q(p)$ is a probability, we then have $q(p)<1$. So, $p_c \leq p_*$. 

To see the reverse inequality, suppose that there exists $p_0 < p_*$ with $q(p_0)= v<1$. \thref{fact:<} and $q(1)=0$ imply that $q\colon [p_*,1] \to [0,1]$ is a continuous bijection. Thus, there is $p_1 > p_*$ with $q(p_1)=v$. As $v<1$, \eqref{eq:h} implies that $h(p_0,v) = 0 = h(p_1,v)$. This contradicts \thref{fact:!}, which requires that $p_0=p_1 = F(v)$. So, $q=1$ for all $p \geq p_*$. Thus, $p_c \geq p_*$. 
\end{proof}

\bibliographystyle{amsalpha}
\bibliography{BA.bib}

\end{document}